\documentclass{amsart}

\usepackage{url}
\usepackage{mathrsfs}

\usepackage{color}
\usepackage{graphicx}
\usepackage{caption}
\usepackage{subcaption}
\usepackage{multirow}
\usepackage{mathtools}
\usepackage{amssymb}
\usepackage{ctable}

\newcommand{\idl}{\mbox{Ideal}}
\newcommand{\re}{\mathbb{R}}

\newcommand{\cpx}{\mathbb{C}}

\newcommand{\F}{\mathbb{F}}
\newcommand{\N}{\mathbb{N}}

\newcommand{\lmd}{\lambda}

\newcommand{\eps}{\epsilon}

\newcommand{\dt}{\delta}
\newcommand{\Dt}{\Delta}
\def\af{\alpha}

\def\rank{\mbox{rank}}

\newcommand{\reff}[1]{(\ref{#1})}

\newcommand{\mc}[1]{\mathcal{#1}}
\newcommand{\mt}[1]{\mathtt{#1}}

\newcommand{\st}{\mathit{s.t.}}

\newcommand{\bdes}{\begin{description}}
	\newcommand{\edes}{\end{description}}
\newcommand{\bal}{\begin{align}}
\newcommand{\eal}{\end{align}}
\newcommand{\bnum}{\begin{enumerate}}
	\newcommand{\enum}{\end{enumerate}}
\newcommand{\bit}{\begin{itemize}}
	\newcommand{\eit}{\end{itemize}}
\newcommand{\bea}{\begin{eqnarray}}
\newcommand{\eea}{\end{eqnarray}}
\newcommand{\be}{\begin{equation}}
\newcommand{\ee}{\end{equation}}
\newcommand{\baray}{\begin{array}}
	\newcommand{\earay}{\end{array}}
\newcommand{\bsry}{\begin{subarray}}
	\newcommand{\esry}{\end{subarray}}
\newcommand{\bca}{\begin{cases}}
	\newcommand{\eca}{\end{cases}}
\newcommand{\bcen}{\begin{center}}
	\newcommand{\ecen}{\end{center}}
\newcommand{\bbm}{\begin{bmatrix}}
	\newcommand{\ebm}{\end{bmatrix}}
\newcommand{\bmx}{\begin{matrix}}
	\newcommand{\emx}{\end{matrix}}
\newcommand{\bpm}{\begin{pmatrix}}
	\newcommand{\epm}{\end{pmatrix}}
\newcommand{\btab}{\begin{tabular}}
	\newcommand{\etab}{\end{tabular}}

\newtheorem{theorem}{Theorem}[section]

\newtheorem{prop}[theorem]{Proposition}
\newtheorem{lem}[theorem]{Lemma}

\newtheorem{cor}[theorem]{Corollary}

\newtheorem{defi}[theorem]{Definition}
\theoremstyle{definition}
\newtheorem{example}[theorem]{Example}
\newtheorem{exm}[theorem]{Example}
\newtheorem{alg}[theorem]{Algorithm}

\numberwithin{equation}{section}

\begin{document}

\title{Loss Functions for Finite Sets}

\author[Jiawang Nie]{Jiawang~Nie}
\address{Jiawang Nie, Department of Mathematics,
University of California San Diego,
9500 Gilman Drive, La Jolla, CA 92093, USA.}
\email{njw@math.ucsd.edu}

\author[Suhan Zhong]{Suhan~Zhong}
\address{Suhan Zhong, Department of Mathematics, 
Texas A\&M University, College Station, TX 77843-3368, USA.}
\email{suzhong@tamu.edu}



\begin{abstract}
This paper studies loss functions for finite sets.
For a given finite set $S$,
we give sum-of-square type loss functions of minimum degree.
When $S$ is the vertex set of a standard simplex,
we show such loss functions have
no spurious minimizers (i.e., every local minimizer is a global one).
Up to transformations, we give similar loss functions
without spurious minimizers for general finite sets.
When $S$ is approximately given by a sample set $T$,
we show how to get loss functions by
solving a quadratic optimization problem.
Numerical experiments and applications are given
to show the efficiency of these loss functions.
\end{abstract}

\maketitle

\section{Introduction}

This paper studies loss functions for finite sets.
The questions of concerns are: for a finite set,
how do we construct a convenient loss function for it?
When does the loss function have no spurious optimizers,
i.e., every local optimizer is also a global one?
We discuss these topics in this paper.
Let $n,k$ be positive integers.
Suppose $S$ is a set of $k$ distinct points in the
$n$-dimensional real Euclidean space $\mathbb{R}^n$.
A function $f$ in $x:=(x_1,\ldots,x_n)$ is
said to be a loss function for $S$ if
the global minimizers of $f$
are precisely the points in $S$.
For convenience, we often select $f$
such that $f$ is nonnegative in $\re^n$
and the minimum value is zero.
Mathematically, this is equivalent to that
\begin{equation}
\label{eq:def:LF:exact}
f(x)=0 \quad \mbox{if and only if} \quad
 x\in S.
\end{equation}
When $S = \{u_1, \ldots, u_k \}$,
a straightforward choice for the loss function is
$f = \|x-u_1\|^2 \cdots \| x - u_k\|^2$,
where $\| \cdot \|$ is the standard Euclidean norm.
This loss function is a polynomial of degree $2k$ in the variable $x$.
It requires to use all points of $S$.
In applications, the cardinality $k$ may be big.
Moreover, the set $S$ often has noises and
it may be given by a large number of samplings
around the points in $S$.
For this reason, the above choice of loss function
may not be convenient in computational practice.

A frequently used loss function is the class of
sum-of-squares (SOS) polynomials. That is, the loss function
$f$ is in the form
\[
f = p_1^2+\cdots+p_m^2,
\]
where each $p_i$ is a polynomial in $x$.
Then $f$ is a loss function for $S$
if and only if each $p_i \equiv 0$ on $S$.
For convenience of computation, we prefer that
$f$ and each $p_i$ have degrees as low as possible.
A more preferable function is that every local minimizer of
$f$ is a global minimizer (i.e., a zero of $f$).
That is, we wish that the loss function $f$
has no {\textit{spurious minimizer}.}\footnote{
A local minimizer that is not a global minimizer is
called a spurious minimizer.}
Optimization without spurious minimizers
is studied in \cite{LassSpur20,LiCaiWei19}.
Polynomial loss functions have good mathematical properties
and are convenient computationally
(see \cite{BabbPolyLoss,GonzalezPolyLF,HuberLoss}).
In particular, polynomial optimization problems
(especially nonconvex ones) can be efficiently solved
by Moment-SOS relaxations. We refer to
\cite{TEiCP,Las01,LasBk15,LasICM,LaurentSOSmom2009,LauICM,locmin,NYZ21}
for recent work on polynomial optimization.

In applications, the set $S$ may not be given explicitly.
It is often approximately given by a sample set
\[
T = \{v_1,\ldots,v_N\},
\]
where each $v_i$ is a sample for a point in $S$
and the sample size $N \gg k$. For such a case,
we can choose a family $\mc{F}$ of loss functions,
which is parameterized to represent the set $S$.
Since $S$ is approximated by $T$,
we choose a loss function $f \in \mc{F}$
such that the average value of $f$ on $T$ is minimum.
Mathematically, this is equivalent to solving the optimization
\begin{equation}
\label{eq:def:LF:inexact}
\min_{f \in \mc{F} }\quad
\frac{1}{N}\sum_{i=1}^N f(v_i) .
\end{equation}
The optimization \reff{eq:def:LF:inexact} requires that
we choose parameters for $f$ such that the average loss on
$T$ is minimum. The set $S$ can be determined by parameters
for $f$ in the family $\mc{F}$.

Loss functions are useful in data science optimization.
There are broad applications of loss functions
\cite{RobustLossFuncBarron,ChengTriLF,Christoff04,GMMwTD21,YHKo05,
LowRankTensorApp17,Schorfheide00,Sudre17,WuShamsu04}.
Selection of loss functions needs to consider application purposes and data structures.
There are various types of loss functions for different applications.
We refer to the survey \cite{WangLossSurvey}
for loss functions in machine learning.
Polynomial loss functions are used in optimal control \cite{Ichihara09,Ito21}.
Linear loss functions are used for network blocking games \cite{LasSze12}.
Loss functions obtained via statistical averaging are given in \cite{BeyAliBew20}.
For inverted beta loss functions, their properties and applications are given in \cite{LeungSpir02}.
Some properties of Erlang loss functions are given in \cite{Jagerman74}.
Properties of correntropic loss functions are given in \cite{SyedPar14}.

\subsection*{Contributions}

The paper studies loss functions for finite sets.
We focus on the SOS type loss functions with minimum degrees.
Let $S$ be a given finite set in $\mathbb{R}^n$.
We characterize loss functions that satisfy \reff{eq:def:LF:exact}.
When $S$ is approximately given by a set $T$ of larger cardinality,
we look for loss functions by
solving the optimization \reff{eq:def:LF:inexact}.
Let $x \coloneqq (x_1,\ldots,x_n)$.
We consider the loss function $f$ such that
$f = p_1^2+\cdots+p_m^2$,
where every $p_i$ is a polynomial in $x$.
The $f$ is a loss function for $S$
if and only if $S$ precisely consists of
common real zeros of polynomials $p_1, \ldots, p_m$.
Mathematically, this is equivalent to that
\be \label{S=commonzeros}
S = \{v\in\re^n: p_1(v) = \cdots = p_m(v) = 0\}.
\ee
For the polynomial $p_i$ to have minimum degrees,
we consider {\it generating polynomials} for the $S$,
which are introduced for symmetric tensor decomposition
\cite{GPandTensorDecomp,LowRankTensorApp17}.
Let $\Phi$ be the set of all generating polynomials for $S$.
It is interesting to note that $\Phi$ has the minimum degree,
such that \reff{S=commonzeros} holds. In particular,
when $S$ is given by vertices of a standard simplex,
the resulting loss function $f$ does not have spurious minimizers.
Up to transformations, we can get loss functions
without spurious minimizers, for general finite sets.
In computational practice, we choose such loss functions of degree four.

When the set $S$ is approximately given by a set $T$ of larger size,
we propose to solve the optimization \reff{eq:def:LF:inexact}
to get the loss function. Equivalently, we determine parameters
for $f$ from a family $\mc{F}$ of loss functions of $S$.
Each $f \in \mc{F}$ is determined by a set of parameters,
and vice versa. By solving \reff{eq:def:LF:inexact},
we not only get a loss function,
but also get a set $S^*$ of $k$ points
that are approximations for the points in $S$.
Once $S^*$ is determined, up to transformations, we can use $S^*$
to get loss functions that have no spurious minimizers.

In summary, our major results are:
\begin{itemize}
	
\item For a given finite set $S$, we give an SOS type
loss function of minimum degree,
such that $S$ is precisely the set of global minimizers.

\item When $S$ consists of the vertices of a standard simplex,
we show that the selected loss function has no spurious minimizers.
For more general finite sets, we give these loss functions
by applying transformations.

\item When the set $S$ is approximately given by a sample set $T$,
we solve the optimization \reff{eq:def:LF:inexact}
to get loss functions of similar properties, 
i.e., they are in SOS type
and have minimum degrees.
\end{itemize}

The paper is organized as follows.
In Section~\ref{sec:prelim}, we briefly review some
backgrounds for polynomial ideals.
In Section~\ref{sec:genframe}, we show how to
get SOS type loss functions for finite sets, with desired properties.
In Section~\ref{sec:standard}, when the set $S$ consists of
vertices of a standard simplex, we show that
the constructed loss functions have no spurious minimizers.
For more general $S$, we show how to get similar loss functions
by applying transformations. In Section~\ref{sc:inexactS},
we show how to get loss functions when the set $S$ is approximately given by a sample set $T$.
Some numerical experiments are given in Section~\ref{sc:num}.

\section{Preliminaries}
\label{sec:prelim}

\subsection*{Notation}

The symbol $\re$ (resp., $\cpx$, $\N$) denotes
the set of real (resp., complex, nonnegative integer) numbers respectively.
The symbol $\mathbb{N}^n$ (resp., $\mathbb{R}^n$, $\mathbb{C}^n$)
stands for the set of $n$-dimensional vectors with entries in $\mathbb{N}$
(resp., $\mathbb{R}$, $\mathbb{C}$) respectively.
For an integer $k>0$, $[k]:=\{1,\cdots,k\}$.
We use $\mathtt{0}$ to denote the vector of all zeros and
$e$ to denote the vector of all ones.
The symbol $e_i$ stands for the unit vector such that
the $i$th entry is one and all other entries are zeros.
For a vector $v$, the $\|v\|$ denotes its Euclidean norm.
For a vector $u\in\mathbb{R}^n$ and $\dt \ge 0$,
$B(u,\dt)  \coloneqq  \{x\in\re^n:\|x-u\|\le \dt\}$
denotes the closed ball centered at $u$ with radius $\dt$.
The symbol $I_n$ denotes the $n$-by-$n$ identity matrix.
The superscript $^T$ (resp., $^\mathtt{H}$)
denotes the operation of matrix transpose (resp., Hermitian).
A square matrix $A$ is said to be positive semidefinite
(resp., positive definite) if $x^TAx\ge 0$
(resp., $x^TAx>0$) for all nonzero vectors $x$.
For two square matrices $X,Y$ of the same dimension,
their commutator is
\[
[X, Y] \, \coloneqq  \, XY-YX.
\]
That is, $X$ commutes with $Y$ if and only if $[X,Y]=0$.
For a function $f$ which is continuously differentiable in $x=(x_1,\ldots, x_n)$,
the $\nabla f$ denotes its gradient in $x$
and $\nabla^2 f$ denotes its Hessian.

Let $\F = \re$ or $\cpx$.
Denote by $\F[x] \coloneqq \F[x_1,\ldots,x_n]$ the ring of polynomials in
$x \coloneqq (x_1,\ldots,x_n)$ with coefficients in $\mathbb{F}$.
For every $d\in\mathbb{N}$, $\mathbb{F}[x]_d$
denotes the subspace of $\mathbb{F}[x]$
which contains all polynomials of degree at most $d$.
For every $\alpha=(\alpha_1,\ldots,\alpha_n)\in\mathbb{N}^n$,
denote the monomial
$
x^{\alpha} \coloneqq x_1^{\alpha_1}\cdots x_n^{\alpha_n} .
$
Its total degree is  $|\alpha| \coloneqq \alpha_1+\cdots+\alpha_n$.

A subset $I\subseteq \mathbb{F}[x]$ is an ideal of $\mathbb{F}[x]$
if $p\cdot q\in I$ for all $p\in I$, $q\in \mathbb{F}[x]$,
and $p_1+p_2\in I$ for all $p_1,p_2\in I$. For an ideal $I$,
its {\it radical} is the set
\[
\sqrt{I} \coloneqq \{f\in \F[x]:\,
f^k \in I\,\text{for some}\,\,  k \in\mathbb{N}\}.
\]
The set $\sqrt{I}$ is also an ideal and $I\subseteq \sqrt{I}$.
The ideal $I$ is said to be radical if $I=\sqrt{I}$.
Each ideal $I$ determines the {\it variety} in $\mathbb{F}^n$ as
\[
V_{\mathbb{F}}(I) \coloneqq \{x\in\mathbb{F}^n:\, p(x)=0\,(p\in I)\}.
\]
For a polynomial tuple $p \coloneqq (p_1,\ldots,p_m)$, we similarly denote that
\[
V_{\mathbb{F}}(p) \, \coloneqq \, \{x\in \F^n:\, p(x)=0\} .
\]
The tuple $p$ generates the ideal
\[
\idl(p)  \coloneqq  p_1\cdot \mathbb{F}[x]+\cdots+p_m\cdot\mathbb{F}[x] .
\]
Clearly, $V_{\mathbb{F}}(\idl(p))=V_{\mathbb{F}}(p)$.

For a set $S\subseteq \mathbb{C}^n$, its vanishing ideal is
\[
I(S)  \coloneqq  \{q\in\mathbb{C}[x]:\, q(u)=0\,( u\in S)\}.
\]
If $S=V_{\mathbb{C}}(p)$ for some polynomial tuple $p$ in $x$,
then $\idl(p)\subseteq I(S)$ but the equality may not hold.
For every $I\subseteq\mathbb{C}[x]$, we have $I(V_{\mathbb{C}}(I))=\sqrt{I}$.
This is Hilbert's Nullstellensatz \cite{IdealVarAlgBook}.

For a given ideal $I\subseteq \mathbb{C}[x]$,
it determines an equivalence relation $\sim$ on $\mathbb{C}[x]$
such that $p\sim q$ if $p-q\in I$, or equivalently, $p\equiv q\mod I$.
Then every $p\in\mathbb{C}[x]$ corresponds to
an equivalence class with the module of $I$, i.e.,
\[
[p]=\{q\in\mathbb{C}[x]:\, q\equiv p\mod I\}.
\]
The set of all equivalent classes is the quotient ring
\[
\mathbb{C}[x]/I  \coloneqq  \{[p]:\,p\in\mathbb{C}[x]\}.
\]

\section{A class of loss functions}
\label{sec:genframe}

In this section, we give a general framework of
constructing loss functions for finite sets.
For convenience, we assume the finite sets are real.
Suppose $S \subseteq \mathbb{R}^n$ is a finite set of cardinality $k$, say,
\[
S \,= \, \{u_1,\ldots, u_k\}.
\]
A function $f$ is a loss function for $S$ if and only if
the global minimizers of $f$ are precisely the points of $S$.
In computational practice,
we often consider the sum-of-squares loss functions
\begin{equation}
\label{eq:Fsos}
f \,= \, p_1^2+\cdots+p_m^2,
\end{equation}
where each $p_i$ is a polynomial in $x$. Denote the tuple
\[
p \,= \, (p_1,\ldots, p_m).
\]
Without loss of generality, one can assume that
the minimum value of $f$ is zero, up to shifting of a constant.
Note that $f(x)=0$ if and only if $p(x)=0$.
Therefore, $f$ is a loss function for $S$ if and only if
\begin{equation}
\label{eq:S=V(p)}
S = \{ x \in \mathbb{R}^n : \, p_1(x)=\cdots=p_m(x)=0 \}.
\end{equation}
The above observation gives the following lemma.

\begin{lem}
\label{lm:f=sumofp}
Let $S, f$ be as above.
Then $f$ is a loss function for $S$
if and only if $S$ is the real zero set of $p$,
i.e., $S = V_{\re}(p)$.
\end{lem}

The existence of $p$ such that $S = V_{\re}(p)$ is obvious.
For instance, one can choose $p_i$ to be a product like
\[
(x_{j_1}-(u_1)_{j_1}) \cdot (x_{j_2}-(u_2)_{j_2})
\cdots (x_{j_k}-(u_k)_{j_k})  ,
\]
for all possible $j_1,\ldots, j_k \in \{1, \ldots, n\}$.
However, for such a choice of $p$, each $p_i$ has degree $k$
and $f$ has degree $2k$. The degree is high if the cardinality $k$ is big,
and there are $n^k$ such products.
This is not practical in applications. In particular, if the set
$S$ is approximately given by a sample set of large size,
then the resulting $p$ is not convenient for usage.
In applications, people prefer loss functions of low degrees.

In the following, we show how to choose
a computationally efficient loss function for $S$.
Let $\mathbb{B}_0$ be the set of first $k$ vectors in the nonnegative power set
$\N^n$, in the graded lexicographic ordering, i.e.,
\begin{equation}
\label{def:B0}
\mathbb{B}_0  \coloneqq  \Big\{
\underbrace{\mathtt{0},\,e_1,\,\ldots,\,e_n,\, 2e_1,\, e_1+e_2,
\ldots,}_{\mbox{first $k$ of them}}
\Big\}.
\end{equation}
Then, we consider the set
\begin{equation}
\label{def:B1}
\mathbb{B}_1 \coloneqq \Big(
 (e_1+\mathbb{B}_0)\cup\cdots\cup (e_n+\mathbb{B}_0)
\Big) \setminus \mathbb{B}_0 .
\end{equation}
For convenience of notation, denote the monomial vectors
\[
[x]_{\mathbb{B}_0} \coloneqq \big( x^\af \big)_{ \af \in  \mathbb{B}_0 }, \quad
[x]_{\mathbb{B}_1} \coloneqq \big( x^\af \big)_{ \af \in  \mathbb{B}_1 } .
\]
Since $S$ is a finite set of cardinality $k$,
we wish to select $\mathbb{B}_0$ so that the set of equivalent classes
of monomials in $\{x^{\beta}: \beta\in\mathbb{B}_0\}$ is a basis for
the quotient space $\mathbb{R}[x]/I(S)$,
where $I(S)$ is the vanishing ideal of $S$.
This requires that $x^{\alpha}$ ($\alpha\in\mathbb{B}_1$)
is a linear combination of monomials $x^{\beta}\,(\beta\in \mathbb{B}_0)$,
modulo $I(S)$. Equivalently, there exist scalars $G(\beta,\alpha)$ such that
\begin{equation}
\label{eq:varphiGa}
\varphi[G,\alpha](x) \coloneqq x^{\alpha}-
\sum_{\beta\in\mathbb{B}_0}G(\beta,\alpha)x^{\beta}
\equiv 0  \mod \, I(S)
\end{equation}
for each $\af \in \mathbb{B}_1$. Let
$G  \coloneqq (G(\beta,\alpha))\in \re^{\mathbb{B}_0\times\mathbb{B}_1}$
be the matrix of all such scalars $G(\beta,\alpha)$.
The polynomial $\varphi[G,\alpha]$
has coefficients that are linear in entries of $G$.
For convenience, denote that
\be   \label{mat:X0X1}
\boxed{
\begin{aligned}
\varphi[G] &= \big(\varphi[G,\alpha] \big)_{\alpha\in\mathbb{B}_1}, \\
X_0 &= \big[ [u_1]_{\mathbb{B}_0} \quad \cdots \quad  [u_k]_{\mathbb{B}_0} \big] ,\\
X_1 &= \big[ [u_1]_{\mathbb{B}_1} \quad \cdots \quad  [u_k]_{\mathbb{B}_1} \big] .
\end{aligned}
}
\ee
The $X_0$ is a square matrix, which is nonsingular
if the points in $S$ are in generic positions.
For $\varphi[G]$ to vanish on $S$,
the equation \reff{eq:varphiGa} implies that
\[
X_1-G^TX_0 = 0.
\]
If $X_0$ is nonsingular, then the matrix $G$ is given as
\begin{equation}
\label{eq:def:G}
G = X_0^{-T}X_1^T.
\end{equation}

We look for conditions on $G$ such that $\varphi[G]$
has $k$ common zeros in $\cpx^n$. For each $i=1,\ldots,n$,
define the multiplication matrix $M_{x_i}(G)$ such that
\begin{equation}
\label{def:M_xMatrix}
[M_{x_i}(G)]_{\mu,\nu}=\left\{
\begin{array}{ll}
1 & \text{if}\quad x_i\cdot x^{\nu}\in\mathbb{B}_0,\,\mu=\nu+e_i,\\
0 & \text{if}\quad x_i\cdot x^{\nu}\in\mathbb{B}_0,\,\mu\not=\nu+e_i,\\
G(\mu,\nu+e_i) & \text{if}\quad x_i\cdot x^{\nu}\in\mathbb{B}_1.
\end{array}
\right.
\end{equation}
The rows and columns of $M_{x_i}(G)$
are labelled by monomial powers $\mu,\nu \in \mathbb{B}_0$.
The following proposition characterizes
when $\varphi[G]$ has $k$ common zeros.

\begin{prop}\label{pro:MG:com}
(\cite[Proposition 2.4]{GPandTensorDecomp})
Let $\mathbb{B}_0,\,\mathbb{B}_1$ be as in (\ref{def:B0})-(\ref{def:B1}).
Then, the polynomial tuple $\varphi[G]$ has $k$ common complex zeros
(counting multiplicities) if and only if
the multiplication matrices $M_{x_1}(G),\ldots, M_{x_n}(G)$ commute, i.e.,
\be \label{eq:Mcommu}
[M_{x_i}(G),M_{x_j}(G)] = 0\quad (1\le i < j\le n) .
\ee
In particular, $\varphi[G]$ has $k$ distinct complex zeros if and only if
$M_{x_1}(G),\ldots, M_{x_n}(G)$ are simultaneously diagonalizable.
\end{prop}

The polynomial tuple $\varphi[G]$
generates the vanishing ideal $I(S)$ of $S$ and
$p=\varphi[G]$ has minimum degrees for \reff{eq:S=V(p)} to hold.

\begin{theorem}  \label{thm:radical}
Assume $S$ is a finite set such that $X_0$ is nonsingular.
Let $G$ be as in (\ref{eq:def:G}).
Then, the ideal $\idl(\varphi[G])$ equals the vanishing ideal of $S$, i.e.,
\be \label{varhiG=I(S)}
\idl(\varphi[G]) \, = \, \{
h \in \re[x]: \, h \equiv 0
\,\,\mbox{on} \, S \} .
\ee
In particular, if a polynomial $h$ vanishes on $S$ identically,
then there are polynomials $p_\af$
($\af \in \mathbb{B}_1$) such that
\be  \label{eq:MinBasis}
h \,= \sum_{ \af \in \mathbb{B}_1 }
q_\af \varphi[G,\af]), \quad
\deg( q_\af ) + |\af|   \le \deg(h).
\ee
\end{theorem}
\begin{proof}
Since $X_0$ is nonsingular, the set $S$ has $k$ distinct points.
Since $G$ is given as in (\ref{eq:def:G}),
the polynomial equation $\varphi[G](x)=0$ has $k$ distinct solutions.
By Proposition~\ref{pro:MG:com}, the multiplication matrices
$M_{x_1}(G),\ldots, M_{x_n}(G)$ are simultaneously diagonalizable.
Note that the ideal $\idl(\varphi[G])$ is zero-dimensional,
because the quotient space $\cpx[x]/\idl(\varphi[G])$
has the dimension $k$.
The ideal $\idl(\varphi[G])$ must be radical.
This can be implied by Corollary~2.7 of \cite{StickelbergerThm}.
So \reff{varhiG=I(S)} holds.

Suppose $h$ is a polynomial such that $h \equiv 0$ on $S$.
Then the above shows that $h \in \idl(\varphi[G])$.
So there exist polynomials $q_\af$
($\af \in \mathbb{B}_1$) such that
\[ h = \sum_{\af \in \mathbb{B}_1} q_\af \varphi[G,\af] . \]
The multiplication matrices
$M_{x_1}(G),\ldots, M_{x_n}(G)$ commute.
One can check that the set of polynomials in the tuple
$\varphi[G]$ is a Gr\"{o}bner basis for $\idl(\varphi[G])$,
with respect to the graded lexicographical ordering.
This can also be implied by the proof of
Lemma~2.8 in \cite{GPandTensorDecomp}.
Therefore, we can further select polynomials
$q_\af \in \re[x]$ with degree bounds as in \reff{eq:MinBasis}.
\end{proof}

The condition that $X_0$ is nonsingular holds
when the points of $S$ are in generic positions.
The equation \reff{eq:MinBasis}
shows that the polynomial tuple $\varphi[G]$
is a minimum-degree generating set for the vanishing ideal $I(S)$.
The following are some examples.

\begin{example}
\label{ex:gen:loss}
i) Consider the set $S$ in $\re^3$ such that
\[
S =
\Big\{\bbm 2\\1\\3\ebm, \bbm -1\\-2\\4\ebm\Big\},
\]
\[
\mathbb{B}_0 = \Big\{
\bbm 0 \\ 0 \\ 0 \ebm, \bbm 1 \\ 0 \\ 0 \ebm
\Big\},\quad
\mathbb{B}_1 = \Big\{
\bbm 0 \\ 1 \\ 0 \ebm, \bbm 0 \\ 0 \\ 1 \ebm,
\bbm 2 \\ 0 \\ 0 \ebm, \bbm 1 \\ 1 \\ 0 \ebm,  \bbm 1 \\  0 \\ 1 \ebm
\Big\}.
\]
The matrix $G$ as in \reff{eq:def:G} and $\varphi[G]$ are
\[
G = \left[\baray{rrrrr}
-1 & \frac{11}{3} & 2 & 2 & -\frac{2}{3}\\
1 & -\frac{1}{3} & 1 & 0 & \frac{10}{3}
\earay\right],  \quad
\varphi[G] =
\left[\begin{aligned}
x_2-x_1+1\\
\frac{x_1}{3}+x_3-\frac{11}{3} \\
x_1^2-x_1-2\\
x_1x_2-2\\
x_1x_3-\frac{10x_1}{3} + \frac{2}{3}
\end{aligned} \right].
\]

\noindent
ii) Consider the set $S$ in $\re^2$ such that
\[
S =
\Big\{\bbm 2\\-1\ebm, \bbm -1\\3\ebm, \bbm -2\\-2\ebm \Big\},
\]
\[
\mathbb{B}_0 = \Big\{
\bbm 0 \\ 0 \ebm, \bbm 1 \\ 0 \ebm, \bbm 0 \\ 1 \ebm
\Big\},\quad
\mathbb{B}_1 = \Big\{
\bbm 2 \\ 0 \ebm, \bbm 1 \\ 1 \ebm, \bbm 0 \\ 2 \ebm
\Big\}.
\]
The matrix $G$ as in (\ref{eq:def:G}) and $\varphi[G]$ are
\[
G = \frac{1}{19}\left[\baray{rrr}
58 & -14 & 82\\
3 & -23 & -20\\
-12 & -22& 23
\earay \right],
\quad
\varphi[G] = \left[\begin{aligned}
x_1^2 + \frac{12x_2}{19} - \frac{3x_1}{19} - \frac{58}{19}  \\
x_1x_2 + \frac{22x_2}{19} + \frac{23x_1}{19} + \frac{14}{19}  \\
x_2^2 - \frac{23x_2}{19} + \frac{20x_1}{19} - \frac{82}{19}
\end{aligned} \right].
\]
	
\noindent
iii) Consider the set $S$ in $\re^2$ such that
\[
S=\Big\{
\bbm 3\\-1\ebm, \bbm -1\\2\ebm , \bbm 2\\1\ebm , \bbm -2\\-1 \ebm\Big\},
\]
\[
\mathbb{B}_0 = \Big\{
\bbm 0 \\ 0 \ebm, \bbm 1 \\ 0 \ebm, \bbm 0 \\ 1 \ebm, \bbm 2 \\ 0 \ebm
\Big\},\quad
\mathbb{B}_1 = \Big\{
\bbm 1\\ 1 \ebm, \bbm 0 \\ 2 \ebm, \bbm 3 \\ 0 \ebm, \bbm 2 \\ 1 \ebm
\Big\}.
\]
The matrix $G$ in (\ref{eq:def:G}) and the polynomial vector $\varphi[G]$ are
\[
G = \left[\baray{rrrr}
 20 & -5 & -36 & 22 \\
 \frac{7}{2} & -\frac{3}{2} & -2 & \frac{9}{2}\\
 -7 & 3 & 12 & -5\\
 -\frac{9}{2} & \frac{3}{2} & 9 & -\frac{11}{2}
\earay\right],
\quad
\varphi[G] = \left[\begin{aligned}
x_1x_2 + \frac{9x_1^2}{2} + 7x_2 - \frac{7x_1}{2} - 20\\
x_2^2 - \frac{3x_1^2}{2} - 3x_2 + \frac{3x_1}{2} + 5  \\
x_1^3 - 9x_1^2 - 12x_2 + 2x_1 + 36  \\
x_1^2x_2 + \frac{11x_1^2}{2} + 5x_2 - \frac{9x_1}{2} - 22
\end{aligned}\right].
\]
\end{example}

For given $S$, the polynomial tuple $\varphi[G]$ with $G$ as in \reff{eq:def:G},
gives the loss function $f=\|\varphi[G]\|^2$
whose global minimizers are precisely the points in $S$.
However, the loss function $f$
may have spurious minimizers.

\begin{exm}
Consider the $S = \Big\{\bbm 5\\-2\ebm,\bbm 4\\3\ebm \Big\}$ in $\re^2$.
The loss function $f=\|\varphi[G]\|^2$ is
\[
f(x) = (x_2+5x_1-23)^2+
(x_1^2-9x_1+20)^2+(x_1x_2+22x_1-100)^2.
\]
Its total gradient $\nabla f$ is
\[
\left[\baray{r}
 4x_1^3 - 54x_1^2 + 2x_1x_2^2 + 88x_1x_2 + 1260x_1 - 190x_2 - 4990 \\
                             2x_2 - 190x_1 + 2x_1^2x_2 + 44x_1^2 - 46
\earay\right]
\]
and its Hessian $\nabla^2 f$ is
\[
\left[\baray{rr}
12 x_1^2 - 108 x_1 + 2 x_2^2 + 88 x_2 + 1260  &   88 x_1 + 4 x_1 x_2 - 190  \\
                 88 x_1 + 4 x_1 x_2 - 190 &           2 x_1^2 + 2
\earay\right].
\]
By checking the optimality conditions $\nabla f(x)=0,\,\nabla^2f(x)\succeq 0$,
we get a local minimizer $(-2.2588, -49.7911)$,
which is not a global one.
\end{exm}

\section{Simplicial loss functions}
\label{sec:standard}

In this section, we study loss functions when $S$
is the vertex set of a standard simplex.
For such a case,
we show that the loss function $f=\|\varphi[G]\|^2$
has no \textit{spurious minimizers}, i.e.,
every local minimizer of $f$ is also a global minimizer.
Moreover, when $S$ is not the vertex set of a standard simplex,
we apply a transformation and get similar loss functions.

\subsection{Simplicial loss functions}
\label{ssc:simplex}

For a vector $a \coloneqq (a_1,\ldots,a_n)$, with each scalar $a_i \ne 0$,
consider the standard simplex vertex set
\begin{equation}
\label{eq:asimplex}
\Dt_n(a) \, \coloneqq  \, \{\mathtt{0},\, a_1e_1,\ldots, a_n e_n\} .
\end{equation}
For the special case that $a=(1,\ldots,1)$, we denote
\begin{equation}
\label{def:standset}
\Dt_n \coloneqq \{\mathtt{0},\, e_1,\ldots,\, e_n\}.
\end{equation}
When the dimension $n$ is clear in the context,
we just write $\Dt = \Dt_n$ for convenience.
In this subsection, we consider the special case that $S = \Dt_n (a)$.
Then the monomial power sets $\mathbb{B}_0$, $\mathbb{B}_1$ are respectively
\[\begin{aligned}
\mathbb{B}_0 &= \{\mathtt{0}, e_1,\ldots, e_n\}, \\
\mathbb{B}_1 &=\{ 2e_1, e_1+e_2,\,\ldots, 2e_n \}.
\end{aligned}
\]
For the matrix $G\in \mathbb{R}^{\mathbb{B}_0\times \mathbb{B}_1}$
given as in \reff{eq:def:G}, we have that
\begin{equation}
\label{eq:standvar}
\boxed{
\baray{lcll}
\varphi[G, 2e_i] &=& x_i^2 - a_ix_i & (i\in[n]), \\
\varphi[G,e_i + e_j] &=& x_ix_j & (i<j).
\earay
}
\end{equation}
The resulting loss function for the set $\Dt_n(a)$ is
\begin{equation}
\label{def:a_simp_loss}
f(x) = \sum_{i=1}^n x_i^2(x_i-a_i)^2+\sum_{1\le i<j\le n} x_i^2x_j^2.
\end{equation}
In particular, the above loss function for $\Dt_n$ is
\be  \label{def:standf}
F(x) \, \coloneqq \, \sum_{i=1}^nx_i^2(x_i-1)^2 +
\sum_{1\le i<j\le n} x_i^2x_j^2.
\ee
A nice property is that the simplicial loss function as in
\reff{def:a_simp_loss} has no spurious minimizers.
\begin{theorem}
\label{thm:standardloss}
Fix nonzero scalars $a_1,\ldots,a_n$, the function $f$
in (\ref{def:a_simp_loss}) has no spurious minimizers, i.e.,
every local minimizer of $f$ is also a global minimizer.
\end{theorem}
\begin{proof}
Suppose $z = (z_1, \ldots, z_n)$ is a local minimizer of $f$.
Then $z$ satisfies the optimality conditions
\[
\nabla f(z) = 0, \quad \nabla^2 f(z) \succeq 0 .
\]
This implies that for $i = 1,\ldots,n$,
\begin{align}
\label{eq:partial_grad:stand1}
\frac{\partial f}{\partial x_i}(z) &=
2z_i\big( 2z_i^2-3a_iz_i+(z^Tz-z_i^2+a_i^2) \big)=0,\\
\label{eq:partial_grad:stand2}
\frac{\partial^2 f}{\partial x_i^2}(z) &=
12z_i^2-12a_iz_i+2(z^Tz-z_i^2+a_i^2)\ge 0.
\end{align}
Denote $\delta_i(z) \coloneqq a_i^2-8(z^Tz-z_i^2)$. The real solutions for (\ref{eq:partial_grad:stand1}) are $z_i=0$ and
\begin{equation}
\label{eq:z_i}
z_i=\frac{3a_i\pm \sqrt{\delta_i(z)}}{4}\quad \text{if}\quad \delta_i(z)\ge 0.
\end{equation}
If each $z_i=0$, then $z=\mathtt{0}$ is a global minimizer.
Suppose some $z_i$ is nonzero, then it satisfies $\delta_i(z)\ge 0$ and $2z_i^2-3a_iz_i+(z^Tz-z_i^2+a_i^2)=0$. So \reff{eq:partial_grad:stand2}
can be reformulated as
\[
\frac{\partial^2 f}{\partial x_i^2}(z) = 8z_i^2-6a_iz_i=2z_i(4z_i-3a_i)\ge 0.
\]
Plug \reff{eq:z_i} into the above inequality.
Since $\sqrt{\delta_{i}(z)} \le |a_i|< |3a_i|$ (note $a_i\not=0$),
\[
z_i =
\bca
\frac{3a_i-\sqrt{\delta_i(z)}}{4} & \text{if} \, a_i < 0, \\
\frac{3a_i+\sqrt{\delta_i(z)}}{4} & \text{if} \, a_i > 0.
\eca
\]
It is clear that $|z_i|\ge|3a_i/4|$.
If $z_i$ is the only nonzero entry of $z$, then $\sqrt{\delta_i(z)}=|a_i|$ and $z=a_ie_i$,
which is a global minimizer. Suppose $z$ has another nonzero entry $z_j$.
By a similar argument, we can get $\delta_j(z)\ge 0$ and $|z_j|\ge |3a_j/4|$.
Note that $2a_i^2-9a_j^2\ge 0$ since
\[
a_i^2-8\cdot\Big|\frac{3a_j}{4}\Big|^2
\ge a_i^2-8z_j^2 \ge \delta_i(z)\ge 0.
\]
Similarly, $2a_j^2-9a_i^2\ge 0$, so
\[
2a_j^2-9a_i^2\ge 2a_j^2-9\cdot \frac{9}{2}a_j^2=-\frac{77}{2}a_j^2\ge 0.
\]
The above holds if and only if $a_j=0$,
which contradicts that all $a_1, \ldots, a_n$ are nonzero.
Therefore, every local minimizer of $f$ is a global minimizer,
i.e., $f$ has no spurious minimizers.

\end{proof}

\subsection{Transformation for general sets}
\label{ssc:trans}

When $S$ is not a simplicial vertex set,
we can still use the function $F$ in \reff{def:standf}
to get new loss functions, up to a transformation.
These new functions have no spurious minimizers.
They are called {\it transformed} simplicial loss functions.
Consider that $S$ is given as
\be  \label{eq:S:gen}
S \, = \, \{u_1,\ldots,u_k\} .
\ee
We discuss the transformation for two different cases.

\subsection*{Case I: $k \le n+1$}

Consider the vertex set of a standard simplex set in $\mathbb{R}^{k-1}$
\[
\Dt_{k-1} = \{\mathtt{0}, \,e_1,\ldots,\, e_{k-1}\}.
\]
The loss function as in \reff{def:standf} for $\Dt_{k-1}$ is
\begin{equation}
\label{eq:stanLF:k-1}
F_{k-1}(z) \, \coloneqq \sum_{i=1}^{k-1}
z_i^2(z_i-1)^2 + \sum_{1\le i<j\le k-1} z_i^2z_j^2 ,
\end{equation}
in the variable $z =(z_1,\ldots,z_{k-1})$.
Consider the linear map
\be \label{LFtrans1}
\ell : \re^{k-1}\rightarrow \re^n, \quad
\ell(e_i) = u_i-u_k,\, i = 1,\ldots,k-1.
\ee
The representing matrix for the linear map $\ell$ is
\be \label{mat:U}
U =  \bbm u_1-u_k & \cdots &  u_{k-1}-u_k \ebm.
\ee
When $u_1, \ldots, u_k$ are in generic positions,
the matrix $U$ has full column rank. Let
\[
U^{\dagger} \, \coloneqq \, (U^T U)^{-1} U^T
\]
be the Pseudo inverse of $U$.
For $x=(x_1,\ldots,x_n)$, consider the loss function
\be  \label{eq:LF:k<n+1}
f(x) = F_{k-1}\big( U^{\dagger}(x-u_k) \big).
\ee
Recall that $\text{Null}(U^{\dagger})$
denotes the null space of  the matrix $U^{\dagger}$.

\begin{theorem}
\label{thm:k<n+1}
Suppose $k\le n+1$ and $\rank \,\, U = k-1$.
Then, the function $f$ as in (\ref{eq:LF:k<n+1})
is a loss function for the set
\[
S+\text{Null}(U^{\dagger}) \, \coloneqq \,
\{x + y:  \, x \in S, U^{\dagger} y = 0 \} .
\]
Moreover, $f$ has no spurious minimizers.
\end{theorem}
\begin{proof}
The function $f$ as in \reff{eq:LF:k<n+1} is nonnegative everywhere.
Note that $f(x)=0$ if and only if $U^{\dagger}(x-u_k) \in \Dt_{k-1}$.
It holds that
\[
\Dt_{k-1} \, = \, \{ U^{\dagger}(x -u_k ): x\in S \}.
\]
For $x\in \mathbb{R}^n$, we have $U^{\dagger}(x-u_k) \in \Dt_{k-1}$
if and only if $x\in S+\text{Null}(U^{\dagger})$. This shows that
$f$ is a loss function for $S+\text{Null}(U^{\dagger})$ in $\mathbb{R}^n$.

The gradient and Hessian of $f$ can be written as
\[
\nabla_x f(x) = (U^{\dagger})^T \nabla_z F_{k-1}(z),\quad
\nabla_x^2 f(x) = (U^{\dagger})^T \nabla_z^2F_{k-1}(z) U^{\dagger}.
\]
Note that $U^{\dagger}$ has full row rank.
If $u$ is a local minimizer of $f$,
then $\nabla_x f(u) = 0$, $\nabla_x^2 f(u) \succeq 0$.
Let $z = U^\dagger (u-u_k)$, then the above implies that
\[
\nabla_zF_{k-1}(z) = 0,\quad \nabla_z^2F_{k-1}(z)\succeq 0.
\]
As in the proof of Theorem \ref{thm:standardloss},
one can show that $z\in \Dt_{k-1}$.
This implies that $z$ is a global minimizer of $F_{k-1}$
and hence $u$ is a global minimizer of $f$.
So $f$ has no spurious minimizers.
\end{proof}

\subsection*{Case II: $k > n+1$}
\label{case:k>n+1}

Let $\omega: \mathbb{R}^n\rightarrow\mathbb{R}^{k-1}$
be the monomial function such that
\be \label{eq:dimenlift}
[x]_{\mathbb{B}_0} = \bbm 1 \\ \omega(x) \ebm,
\ee
where $\mathbb{B}_0$ is the power set in (\ref{def:B0}).
For the set $S$ as in (\ref{eq:S:gen}), denote
\be  \label{eq:hatS}
\hat{S} \, \coloneqq \, \Big\{\omega(u_1),
\ldots,\omega(u_k) \Big\}
\subseteq\mathbb{R}^{k-1}.
\ee
Define the linear map $\mc{L}$ such that
\[
\mc{L} : \, \mathbb{R}^{k-1} \to \re^{ k-1  }, \quad
\mc{L}(e_i) = \omega(u_i)-\omega(u_k),
\, i = 1,\ldots,k-1.
\]
The representing matrix for the linear map $\mc{L}$ is
\be \label{mat:L}
L = \bbm
\omega(u_1)
& \cdots &
\omega(u_{k-1})
\ebm - \bbm
 \omega(u_k)
& \cdots &
 \omega(u_k)
\ebm  .
\ee
When $u_1,\ldots,u_n$ are in generic positions,
the matrix $L$ is nonsingular. For such a case,  define the function
\be \label{eq:LF:k>n+1:k-1}
\hat{f}(z) \, \coloneqq \,
F_{k-1}\big( L^{-1}(z - \omega(u_k) \big),
\ee
in the $z = (z_1,\ldots,z_{k-1})$,
where $F_{k-1}$ is the simplicial loss function as in (\ref{eq:stanLF:k-1}).
The above $\hat{f}$ is called a transformed
simplicial loss function for $\hat{S}$.
The following theorem follows from Theorem \ref{thm:k<n+1}.

\begin{theorem}
\label{thm:k>n+1}
Suppose $k>n+1$ and $L$ is nonsingular.
Then, the function $\hat{f}$ as in \reff{eq:LF:k>n+1:k-1}
is a loss function for $\hat{S}$ and it has no spurious minimizers.
\end{theorem}
For $x=(x_1,\ldots,x_n)$, define the function
\be \label{eq:LF:k>n+1}
f(x) = F_{k-1}\big(L^{-1}(\omega(x)-\omega(u_k)\big).
\ee
\begin{cor}
\label{cor:k>n+1}
Suppose $k>n+1$ and $L$ in (\ref{mat:L}) is nonsingular,
then the function $f$ in (\ref{eq:LF:k>n+1})
is a loss function for $S$.
\end{cor}
\begin{proof}
The function $f$ as in \reff{eq:LF:k>n+1} is nonnegative everywhere.
By Theorem~\ref{thm:k>n+1}, we know $f(x)=0$
if and only if $\omega(x)\in \hat{S}$.
Since $\omega$ is a one-to-one map, the $f$ is a loss function for $S$.
\end{proof}

The transformed simplicial loss functions in
\reff{eq:LF:k<n+1} and \reff{eq:LF:k>n+1:k-1}
have no spurious minimizers. The following are some
examples of transformed simplicial loss functions.

\begin{example}
\label{ex:trans:loss}
i) Consider the set $S$ in $\re^3$ such that
\[
S = \Big\{\bbm 4\\-2\\1\ebm, \bbm-1\\3\\-5\ebm\Big\} .
\]
The matrix $U$ as in \reff{mat:U} and its Pseudo inverse are
\[
U = \left[\baray{r} 5\\-5\\6 \earay\right],\quad
U^{\dagger} = \frac{1}{86} \left[\baray{r}  5 \\ -5 \\ 6 \earay\right]^T.
\]
Since $k=2$, the simplicial loss function for $\Dt_{k-1}$ is
$F_1 = z^2(z-1)^2$ in the univariate variable $z$.
Then, the transformed simplicial loss function as in (\ref{eq:LF:k<n+1}) is
\[
f(x) = \left(\frac{5x_1}{86}-\frac{5x_2}{86}+\frac{3x_3}{43}+\frac{25}{43}\right)^2
\cdot \left(
\frac{5x_1}{86}-\frac{5x_2}{86}+\frac{3x_3}{43}-\frac{18}{43}
\right)^2.
\]	

\noindent
ii) Consider the set $S$ in $\re^2$ such that
\[S = \Big\{\bbm 2\\3\ebm, \bbm -1\\-2\ebm,
\bbm 1\\-3\ebm, \bbm -2\\ 2\ebm\Big\} .
\]
Since $k=4>n+1$, the set $\hat{S}$ in \reff{eq:hatS} is
\[
\hat{S} = \left\{
\bbm 2\\3\\4\ebm,
\bbm -1\\-2\\1\ebm,
\bbm 1\\-3\\1\ebm,
\bbm -2\\2\\4\ebm
\right\}.
\]
The matrix $L$ as in \reff{mat:L} and its inverse are
\[
L = \left[\baray{rrr}
4  &   1  &   3\\
1  &  -4  &  -5\\
0  &  -3  &  -3
\earay\right],
\quad
L^{-1} = \frac{1}{18}
\left[\baray{rrr}
3 & 6 & -7\\
-3 & 12 & -23\\
3 & -12 & 17
\earay\right].
\]
Since $k=4$, the simplicial loss function for $\Dt_{k-1}$ is
\[
F_3(z) = z_1^2(z_1-1)^2+z_1^2z_2^2+z_2^2(z_2-1)^2
+z_2^2z_3^2+z_3^2(z_3-1)^2.
\]
in the variable $z=(z_1,z_2,z_3)$.
Then, the transformed simplicial loss function
as in (\ref{eq:LF:k>n+1:k-1}) is
$\hat{f}(z) = F_3(L^{-1}(z-\omega(u_4))$, with
\[
L^{-1}(z-\omega(u_4)) = \frac{1}{18}
\left[\baray{r}
3z_1 + 6z_2 - 7z_3 + 22\\
-3z_1 + 12z_2 - 23z_3 +62\\
3z_1 - 12z_2 + 17z_3-38
\earay\right].
\]
\end{example}

\section{Finite sets with noises}
\label{sc:inexactS}

In this section, we study loss functions
for finite sets that are given with noises.
In many applications, the finite set $S$, with the cardinality $k$,
is often approximately given by another finite set
$T$, with the cardinality $N \gg k$.
For instance, each point of $S$ is often approximated by
a number of samplings, and $T$ consists of all such samplings.
The cardinality $N$ is the total number of samplings.
We look for good loss functions for such approximately given sets.
This kind of questions have important applications
in clustering and classification.

\subsection{Best approximation sets}
\label{subsec:deterRep}

Suppose $S$ is approximately given by a sampling set $T$, say,
\be \label{eq:inexactS}
T = \{v_1,\ldots,v_N\}.
\ee
Each point of $S$ is sampled by a certain number of points in $T$.
We discuss how to recover the $k$ points of $S$ from sampling points in $T$.

A finite set can be represented as the optimizer set of a loss function.
For convenience, we consider loss functions
whose minimum values are zeros.
Let $\mc{F}$ be a family of loss functions such that
each $f \in \mc{F}$ has $k$ common zeros.
The loss function family $\mc{F}$ is parameterized by some parameters.
For such given $\mc{F}$, we look for the best loss function in $\mc{F}$
such that its average value on $T$ is the smallest.
This leads to the following definition.

\begin{defi}  \label{def:bestRep}
Let $\mc{F}$ be a family of loss functions such that
each $f \in \mc{F}$ is nonnegative and it has $k$ common zeros.
A set $S^* = \{u_1^*,\ldots,u_k^*\}$ is called the best
$\mc{F}$-approximation set for $T$ as in \reff{eq:inexactS} if
$S^*$ is the zero set of $f^*$, where
$f^*$ is the minimizer of the optimization
\be \label{eq:bestRep}
\left\{
\begin{array}{rl}
\min  & \mu(f) \coloneqq
           \frac{1}{N}\sum\limits_{i=1}^N  f(v_i)\\
\st &  f \in \mc{F} .
\end{array}
\right.
\ee
\end{defi}

For a given set $S$, if the matrix $G$ is as in \reff{eq:def:G},
then $S$ is the common zero set of the polynomial tuple
$\varphi[G]$, given as in \reff{eq:varphiGa}.
In fact, $\idl( \varphi[G] )$ is the vanishing ideal $I(S)$
and $\varphi[G]$ gives the minimum-degree generating set for $I(S)$.
The relation between $S$ and $\varphi[G]$
is characterized by Theorem~\ref{thm:radical}.
As shown in Proposition~\ref{pro:MG:com},
$\varphi[G]$ has $k$ common zeros
(counting multiplicities and all complex ones)
if and only if the multiplication matrices
$M_{x_1}(G),\ldots,M_{x_n}(G)$ commute with each other.
Moreover, $\varphi[G]$ has $k$ distinct zeros if and only if
$M_{x_1}(G),\ldots,M_{x_n}(G)$ are simultaneously diagonalizable.
So, one can use the matrix $G$ and the polynomial tuple
$\varphi[G]$ to represent the finite set $S$.
As in Section~\ref{sec:genframe},
we consider the family of the following loss functions
\be \label{par:fG}
f_G \, \coloneqq \, \|\varphi[G]\|^2 ,
\ee
parameterized by $G$.
We look for the matrix $G$ such that the average of the values of $f_G$ on $T$
is minimum and $\varphi[G]$ has $k$ common zeros.

In view of the above, we consider the following matrix optimization problem
\begin{equation}
\label{eq:bestRep:rel}
\left\{
\begin{array}{rl}
\min   & \vartheta(G)  \coloneqq
          \frac{1}{N}\sum\limits_{j=1}^N f_G(v_j) \\
\st & [M_{x_i}(G),M_{x_j}(G)]=0 \, (1\le i<j\le n).
\end{array}
\right.
\end{equation}
The value $\varphi[G](v_i)$ is linear in the matrix $G$.
The feasible set of \reff{eq:bestRep:rel} is given by a set of
quadratic equations. The optimization \reff{eq:bestRep:rel}
is the specialization of \reff{eq:bestRep}
such that $\mc{F}$ is the family of loss function $f_G$,
with $\varphi[G]$ having $k$ common zeros.

\subsection{Approximation analysis}
\label{subsc:SolAna}

Suppose $G^*$ is the minimizer of \reff{eq:bestRep:rel}.
Let $S_0$ denote the common zero set of $\varphi[G^*]$.
We can use $S_0$ to approximate the points in $S$.
In some applications, the set $S$ contains only real points
and people like to get a real set approximation for $S$.

First, we study the approximation quality of
the optimization~\reff{eq:bestRep:rel}.
For each $\af \in \mathbb{B}_1$, the sub-Hessian
of the objective $\vartheta(G)$
with respect to the $\af$th column $G(:,\af)$ is the matrix
\[
H \, \coloneqq \, \frac{2}{N}\sum_{j=1}^N
[v_j]_{\mathbb{B}_0} ([v_j]_{\mathbb{B}_0})^{\mt{H}}.
\]
In the above, the superscript $^\mt{H}$
denotes the Hermitian transpose.

\begin{theorem}
\label{thm:real_roots:suf}
Let $T$ be as in \reff{eq:inexactS}
and let $S = \{u_1,\ldots,u_k\}$ be such that the matrix
$X_0$ as in \reff{mat:X0X1} is nonsingular.
Assume there exists $\dt >0$ such that
$H \succeq 2\dt I_k$.
Suppose the set $T$ is such that
\be  \label{eq:real_roots:suf}
T \subseteq S + B(0,\eps), \quad
T \cap  B(u_i,\epsilon)\not = \emptyset \,
(i =1,\ldots, k),
\ee
for some $\eps > 0$. Then, as $\epsilon \to 0$,
the optimizer $G^*$ of \reff{eq:bestRep:rel}
converges to $\hat{G} \coloneqq X_0^{-T} X_1^T$,
and the common zero set $S_0$ of $\varphi[G^*]$ converges to $S$.

In particular, when $S, T \subseteq \re^n$,
if $\epsilon > 0$ is sufficiently small,
the common zero set $S_0$
contains $k$ distinct real points.
\end{theorem}
\begin{proof}
First, we show the convergence $G^* \to \hat{G}$ as $\eps \to 0$.
Since the set $\hat{B} \coloneqq \cup_{i=1}^k B(u_i, 1)$ is compact,
the polynomial function $\varphi[\hat{G}](x)$
is Lipschitz continuous on $\hat{B}$.
There exists $R >0$ such that for all $i\in[k]$ and for all $x\in B(u_i, \eps)$,
\[
\| \varphi[\hat{G}](x)-\varphi[\hat{G}](u_i) \|
\le R \|x-u_i\|\le R \epsilon.
\]
Since $T \subseteq S + B(0,\eps)$,
each $v_j \in T$ belongs to some $B(u_{i_j}, \eps)$ for $i_j\in\{1,\ldots,k\}$.
So the above inequality implies that
(note that each $\varphi[\hat{G}](u_{i_j}) =0$)
\[\begin{aligned}
\vartheta (\hat{G}) &= \frac{1}{N} \sum_{j=1}^N
\|  \varphi[\hat{G}](v_j) \|^2\\
&= \frac{1}{N} \sum_{j=1}^N
\|  \varphi[\hat{G}](v_j) - \varphi[\hat{G}](u_{i_j}) \|^2
\le   (R \epsilon)^2 .
\end{aligned}\]
Since $G^*$ is the minimizer of \reff{eq:bestRep:rel}, we have
\be \label{eq:muGstarbd}
0\le \vartheta(G^*)\le \vartheta(\hat{G})\le (R \epsilon)^2.
\ee
Moreover, it holds that
\[
\begin{aligned}
\vartheta (G^*) &=
\frac{1}{N} \sum_{j=1}^N
\|  \varphi[G^*](v_j) - \varphi[\hat{G}](v_j)
+ \varphi[\hat{G}](v_j) \|^2,\\
& \ge
\frac{1}{N} \sum_{j=1}^N
\Big(
\|  \varphi[G^*](v_j) - \varphi[\hat{G}](v_j)  \| - \|   \varphi[\hat{G}](v_j) \|
\Big)^2\\
& \ge
\frac{1}{N^2} \Big( \sum_{j=1}^N
 \|  \varphi[G^*](v_j) - \varphi[\hat{G}](v_j)  \| -
\sum_{j=1}^N\|   \varphi[\hat{G}](v_j) \|
\Big)^2 .
\end{aligned}
\]
In the above, the first inequality follows from that
$\| a+b \|^2 \ge (\|a\|-\|b\|)^2$ and the second inequality follows from
the Cauchy-Schwartz inequality. Then, we have
\[
\sum_{j=1}^N
\|  \varphi[G^*](v_j) - \varphi[\hat{G}](v_j)  \|
\le
N \sqrt{ \vartheta (G^*) } + \sum_{j=1}^N
 \|   \varphi[\hat{G}](v_j) \|
\]
By the formula of $\varphi[G](x)$ and using Cauchy-Schwartz inequality again, we get
\[
\sum_{j=1}^N
\|  (G^* - \hat{G} )^T [v_j]_{\mathbb{B}_0}  \|
\le  N  \Big( \sqrt{  \vartheta (G^*) }+ \sqrt{ \vartheta (\hat{G}) } \Big).
\]
Since $
\sum_{j=1}^N
\|  (G^* - \hat{G} )^T [v_j]_{\mathbb{B}_0}  \|^2
  \le
\big(
\sum_{j=1}^N
\|  (G^* - \hat{G} )^T [v_j]_{\mathbb{B}_0}  \|
\big)^2$, we have
\[
\frac{1}{N} \sum_{j=1}^N
\|  (G^* - \hat{G} )^T [v_j]_{\mathbb{B}_0}  \|^2
\le  N \Big( \sqrt{  \vartheta (G^*) } + \sqrt{ \vartheta (\hat{G}) } \Big)^2.
\]
By the assumption $H \succeq 2\dt I_k$, the above implies
\[
\| G^* - \hat{G} \|
\le  \sqrt{ \frac{N}{\dt} }
\Big( \sqrt{  \vartheta (G^*) }  + \sqrt{ \vartheta (\hat{G}) } \Big).
\]
Therefore, as $\eps \to 0$, we have
$G^*$ converges to $\hat{G}$.

In the following, we assume that $S, T \subseteq \re^n$.
Since $X_0$ is nonsingular, $S$ has $k$ distinct real points.
Recall the multiplication matrices $M_{x_i}(G^*),M_{x_i}(\hat{G})$
given as in (\ref{def:M_xMatrix}).
Since $G^* \to \hat{G}$, the common zero set of
$\varphi[G^*]$ converges to that of $\varphi[\hat{G}]$.
The zero set of $\varphi[\hat{G}]$ is $S$,
which consists of $k$ distinct real points.
Hence, $\varphi[G^*]$ also has $k$ distinct common zeros
when $\epsilon>0$ is sufficiently small.
Then it remains for us to show that all common zeros of $\varphi[G^*]$ are real.
For a vector $\xi = (\xi_1,\ldots,\xi_n)$,
define the matrices
\[
M_1 = \sum_{i=1}^n \xi_i M_{x_i}(G^*),\quad
M_2 = \sum_{i=1}^n \xi_i M_{x_i}(\hat{G}).
\]
Their characteristic polynomials are
\[
p_1(\lambda) \coloneqq \det(M_1-\lambda I),\quad
p_2(\lambda) \coloneqq \det(M_2-\lambda I).
\]
Fix a generic real value for $\xi$ so that $M_2$ has
$k$ distinct real eigenvalues.
This is because $\varphi[\hat{G}](x)$ has real distinct solutions and
by the Stickelberger's Theorem (see \reff{eq:extractS0}
as in \cite{LaurentSOSmom2009,StickelbergerThm}).
Note that both $p_1(\lmd)$, $p_2(\lmd)$ have degree $k$
and all coefficients are real. The $p_2(\lambda)$ has $k$ distinct real roots.
They are ordered as
\[
\hat{\lambda}_1<\hat{\lambda}_2<\cdots<\hat{\lambda}_k .
\]
We can choose real scalars $b_0,\ldots,b_k$ such that
\[
b_0<\hat{\lambda}_1<b_1<\cdots<b_{k-1}<\hat{\lambda}_k<b_k.
\]
As $\eps \to 0$, the coefficients of $p_1$ converge to those of $p_2$.
So, when $\epsilon > 0$ is small enough,
$p_1(b_j)$ has the same sign as $p_2(b_j)$ does.
Since each $p_2(b_{j-1})p_2(b_j)<0$,  we have
\[
p_1(b_{j-1})p_1(b_j)<0, \quad j = 1, \ldots, k+1.
\]
This implies that $p_1$ has $k$ distinct real roots.
Equivalently, $M_1$ has $k$ distinct real eigenvalues for
$\eps > 0$ sufficiently small. By Proposition~\ref{pro:MG:com},
the multiplication matrices
$M_{x_1}(G^*), \ldots, M_{x_n}(G^*)$
are simultaneously diagonalizable.
Also note that $M_1$ is diagonalizable
and there is a unique real eigenvector (up to scaling)
for each real eigenvalue.
This shows that
$M_{x_1}(G^*), \ldots, M_{x_n}(G^*)$
can be simultaneously diagonalized
by common real eigenvectors.
All $M_{x_1}(G^*), \ldots, M_{x_n}(G^*)$
have real entries,
so they have only real eigenvalues.
Therefore, by Stickelberger's Theorem,
$\varphi[G^*]$ has $k$ distinct real common zeros if
$\epsilon>0$ is sufficiently small.
\end{proof}

\subsection{Loss functions for noisy sets}

When the set $S$ is approximately given by the sampling set $T$,
we can solve \reff{eq:bestRep:rel} for an optimizer matrix $G^*$,
to get loss functions. Let $S_0$ be the common zero set
of the polynomial tuple $\varphi[G^*]$.
If $T$ is far from $S$, $S_0$ may have non-real points.
If real points are wanted, we can choose the real part set
\be \label{eq:bestRep:app}
S^{re}  \, \coloneqq \, \{\mbox{Re}(u): \, u \in S_0 \}.
\ee

First, we show how to compute the common zero set $S_0$.
By Stickelberger's Theorem
(see \cite{LaurentSOSmom2009,StickelbergerThm}),
the set $S_0$ can be expressed as
\begin{equation}
\label{eq:extractS0}
S_0 = \left\{(\lambda_1,\ldots, \lambda_n)\left|\begin{array}{c}
\exists q \in\mathbb{C}^k\setminus \{0\}\text{ such that}\\
M_{x_i}(G^*)q = \lambda_i q ,\, i=1,\ldots, n
\end{array}
\right.
\right\}.
\end{equation}
To get $S_0$ numerically, people often use Schur decompositions. Let
\begin{equation}
\label{def:M(xi)}
M_1 = \xi_1 M_{x_1}(G^*)+\cdots+ \xi_n M_{x_n}(G^*),
\end{equation}
where $\xi_1,\ldots, \xi_n$ are generically chosen scalars.
Then, compute the Schur decomposition for $M_1$:
\be \label{Schur:comp}
Q^{\mathtt{H}} M_1 Q \, = \, P,\quad
Q= \bbm q_1 & \cdots & q_k \ebm .
\ee
In the above, $Q \in \cpx^{k \times k}$ is a unitary matrix
and $P  \in \cpx^{k \times k}$ is upper triangular.
Based on the Schur decomposition \reff{Schur:comp},
the common zeros $\hat{u}_1,\ldots,\hat{u}_k$ of $\varphi[G^*]$ can be given as
\be \label{hat:u_i}
\hat{u}_i \coloneqq
\big(q_i^{\mathtt{H}}M_{x_1}(G^*)q_i,\ldots,q_i^{\mathtt{H}}M_{x_n}(G^*)q_i\big),
\quad i = 1,\ldots,k .
\ee
We refer to \cite{CorlessSchurFactor} for how to use
Schur decompositions to compute common zeros
of zero-dimensional polynomial systems.
For general cases, the set $S_0$ contains $k$ distinct points. It holds when $S,T\subseteq\mathbb{R}^n$ and  the points in $T$ are close to $S$;
see Theorem~\ref{thm:real_roots:suf}.

Based on the above discussions,
we get the following algorithm for obtaining loss functions
when $S$ is approximately given by the sampling set $T$.

\begin{alg}
\label{alg:LF}
For the given set $T$ as in \reff{eq:inexactS}
and the cardinality $k$, do the following:

\begin{itemize}

\item[Step 1]
Solve quadratic optimization (\ref{eq:bestRep:rel})
for the optimizer $G^*$.
	
\item[Step 2]
Compute the common zero set
$S_0 = \{\hat{u}_1, \ldots, \hat{u}_k\}$ of $\varphi[G^*]$.
Let $S^*$ be the set $S_0$ or $S^{re}$ be as in \reff{eq:bestRep:app}
if the real points are wanted.
	
\item[Step 3]
Get a loss function for the set $S^*$,
by the method in Section~\ref{sec:genframe}
or Section~\ref{sec:standard}.

\end{itemize}
\end{alg}

In Step~1, the optimization \reff{eq:bestRep:rel} has
a convex quadratic objective, but its constraints are given by
quadratic equations, in the matrix variable $G$.
So \reff{eq:bestRep:rel} is a quadratically constrained
quadratic program (QCQP).
It can be solved as a polynomial optimization problem
(e.g., by the software \texttt{GloptiPoly 3} \cite{henrion2009gloptipoly}).
The classical nonlinear optimization methods,
(e.g., Gauss-Newton, trust region, and Levenberg-Marquardt type methods)
can also be applied to solve \reff{eq:bestRep:rel}.
We refer to \cite{Kel95,More78,yyx11} for such references.

In Step~2, the common zero set $S_0$ can be computed as in \reff{hat:u_i},
by using the Schur decomposition \reff{Schur:comp}
for the matrix $M_1$ in \reff{def:M(xi)},
for generically chosen scalars $\xi_1,\ldots,\xi_n$.

In Step~3,  there are two options for obtaining loss functions
for the set $S^*$, given in
Sections~\ref{sec:genframe} and \ref{sec:standard} respectively.
One is to choose $f=\|\varphi[G]\|^2$;
the other one is to apply a transformation first
and then choose $f$ similarly.
After the transformation, there are no spurious optimizers
for the loss function.

\section{Numerical Experiments}
\label{sc:num}

In this section, we present numerical experiments for loss functions.
The computation is implemented in {\tt MATLAB} R2018a,
in a Laptop with CPU 8th Generation Intel® Core™ i5-8250U and RAM 16 GB.
The optimization problem (\ref{eq:bestRep:rel})
can be solved by the polynomial optimization software~\texttt{GloptiPoly~3}
(with the SDP solver \texttt{SeDuMi}),
or it can be solved by classical nonlinear optimization solvers
(e.g., the {\tt MATLAB} function \texttt{fmincon}
can be used for convenience).

First, we explore the numerical performance of Algorithm \ref{alg:LF}.
\begin{example}
\label{ex:AlgPer}
Consider the set
\[
S = \Big\{
\bbm 1\\1 \ebm, \bbm 3\\2\ebm, \bbm1.5\\2.5\ebm,
\bbm 2.5\\3 \ebm, \bbm 2\\1.5\ebm, \bbm 3\\1\ebm
\Big\}.
\]
Suppose $T$ is a sampling set of $S$ such that
\[
T\subseteq S+\epsilon[-1,1]^2,\quad \mbox{and}
\]
\[
|T\cap \{u_i+\epsilon[-1,1]^2\}| = N_i\,(i=1,\ldots,6).
\]
We apply Algorithm \ref{alg:LF} for cases
$N_i \in \{50,100\}$ and $\epsilon\in \{0.05,0.1,0.5\}$.
The samples are generated with \texttt{MATLAB} function \texttt{randn}.
We summarize the computational results in
Table~\ref{tab:ex:AlgPer} and Figure~\ref{fig:ex:AlgPer}.
In~Table \ref{tab:ex:AlgPer}, the symbol $S^*$
denotes the computed approximation set as in (\ref{eq:bestRep:app}).
We use the distance
\[
\|S-S^*\|  \, \coloneqq \,  \max_{v\in S^*}
\min_{u_i\in S} \, \|v-u_i\|
\]
to measure the approximation quality of $S^*$ to $S$.
The loss function for $S^*$ is in form of $f=\|\varphi[G]\|^2$,
whose maximum value on $S$ is shown in the fourth column.
In Figure~\ref{fig:ex:AlgPer}, the sampling points in $T$
are plotted in dots, the points in $S$ are plotted
in diamonds and the points in $S^*$ are plotted in squares.
The left column from top to bottom shows cases for $N_i=50$ and
$\epsilon = 0.05, 0.1, 0.5$ respectively.
The right column shows cases for $N_i=100$ accordingly.
\begin{table}[htb]
\caption{The numerical results of Example \ref{ex:AlgPer}}
\label{tab:ex:AlgPer}
\begin{tabular}{cccc}
\specialrule{.1em}{0em}{0.1em}
$N_i$ & $\epsilon$ & $\|S-S^*\|$ & $\max\limits_{u \in S} f(u)$ \\
\hline
\multirow{3}{*}{50} & 0.05 & 0.0064 & $1.27\cdot10^{-4}$
\\ \cline{2-4}
 & 0.1 & 0.0145 & $2.98\cdot10^{-4}$
\\ \cline{2-4}
 & 0.5 & 0.1821 & 0.0862\\
\specialrule{.1em}{0em}{0.1em}
\end{tabular}
\quad
\begin{tabular}{cccc}
\specialrule{.1em}{0em}{0.1em}
$N_i$ & $\epsilon$ & $\|S-S^*\|$ & $\max\limits_{u \in S} f(u)$ \\
\hline
\multirow{3}{*}{100} & 0.05 & 0.0055 & $8.06\cdot10^{-5}$\\ \cline{2-4}
 & 0.1 & 0.0067 & $1.89\cdot10^{-4}$\\ \cline{2-4}
 & 0.5 & 0.1080 & 0.0359\\
\specialrule{.1em}{0em}{0.1em}
\end{tabular}
\end{table}
\begin{figure*}[htb]
 \includegraphics[
 width=.5\textwidth]{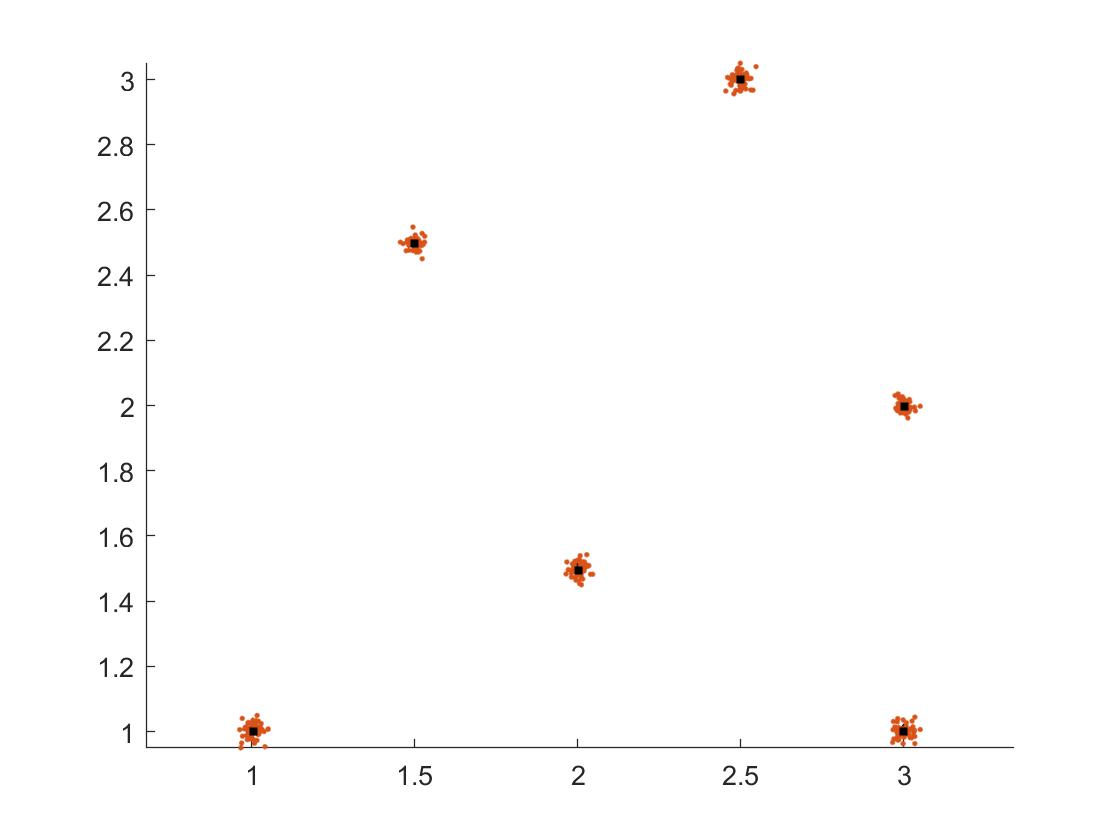}\hfill
 \includegraphics[
 width=.5\textwidth]{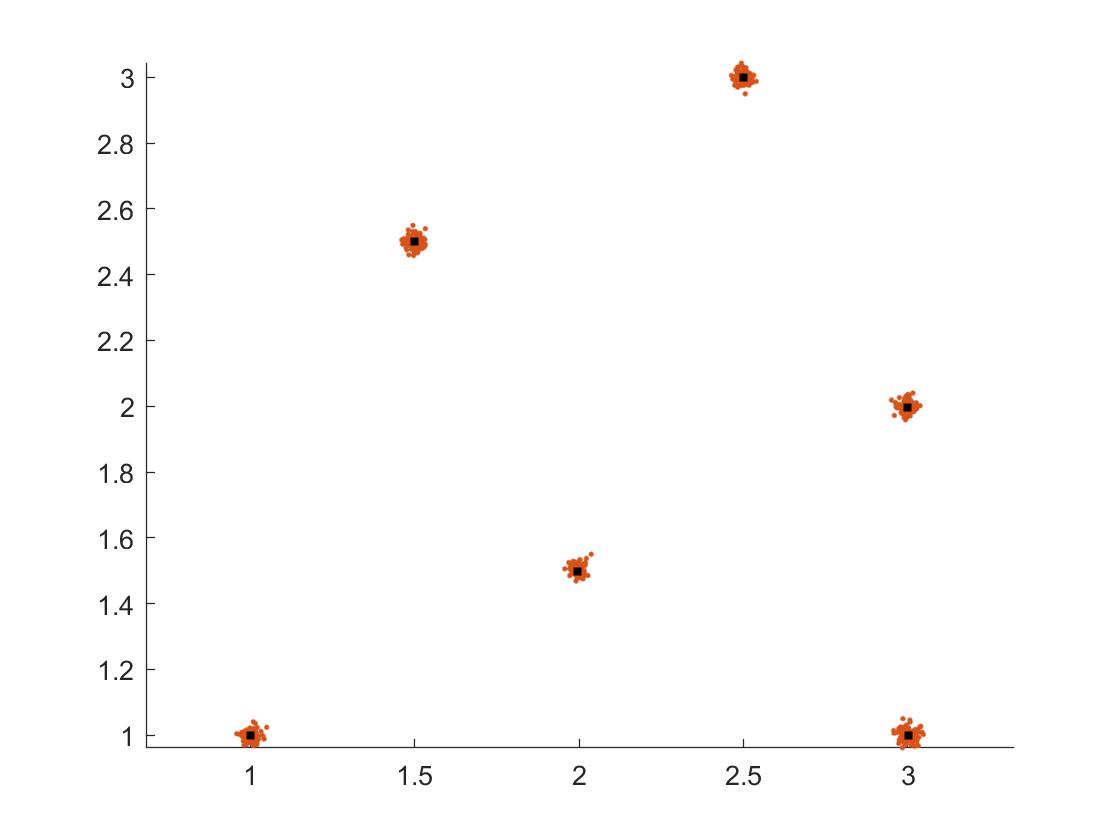}\hfill
 \includegraphics[
 width=.5\textwidth]{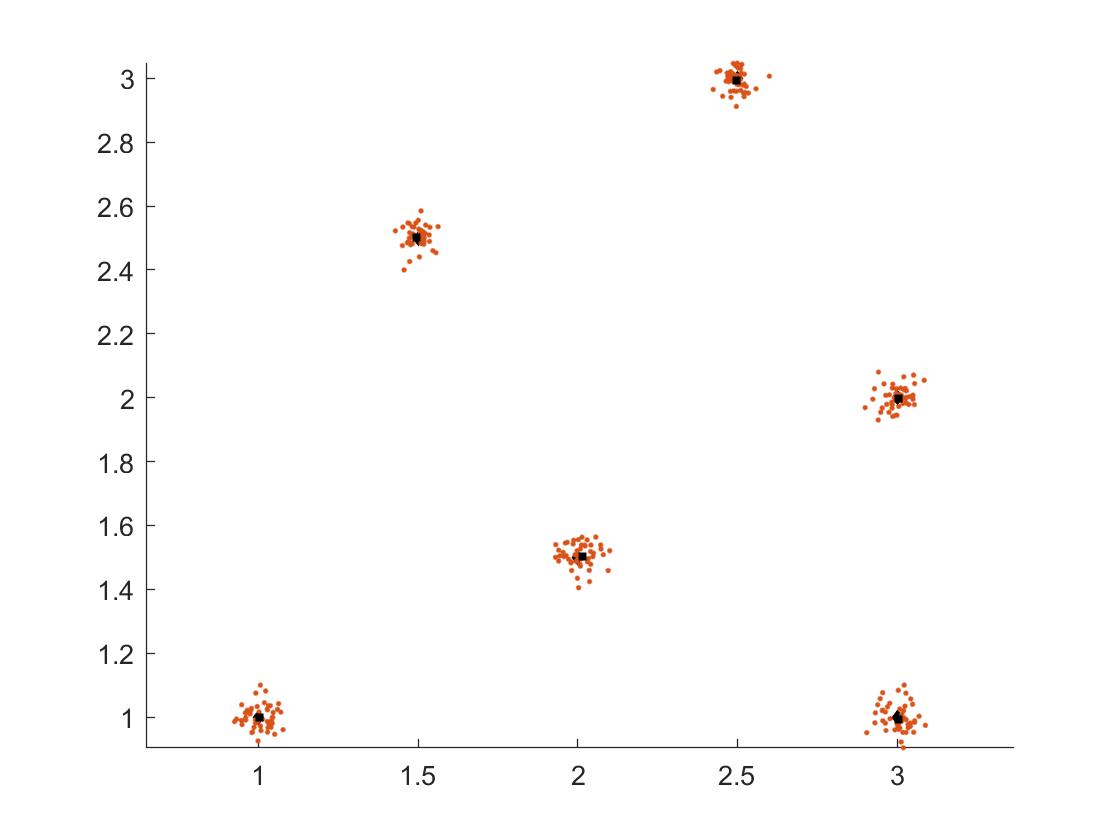}\hfill
 \includegraphics[
 width=.5\textwidth]{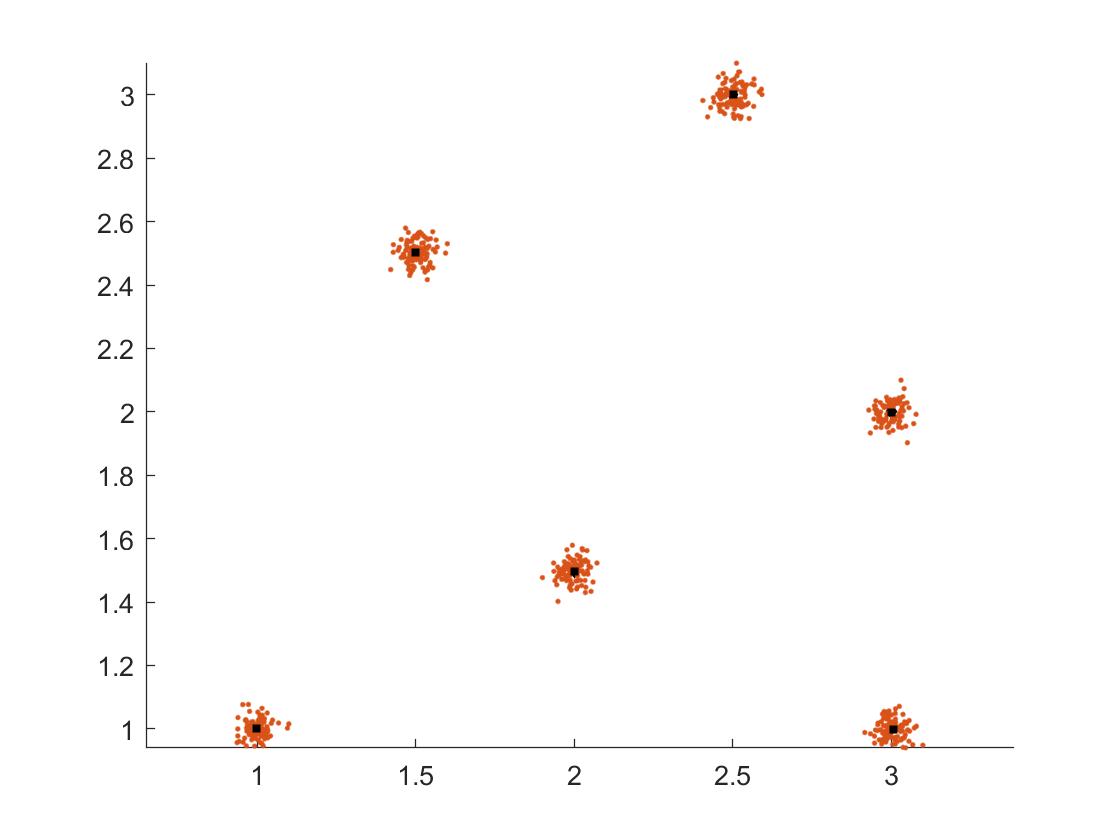}\hfill
 \includegraphics[
 width=.5\textwidth]{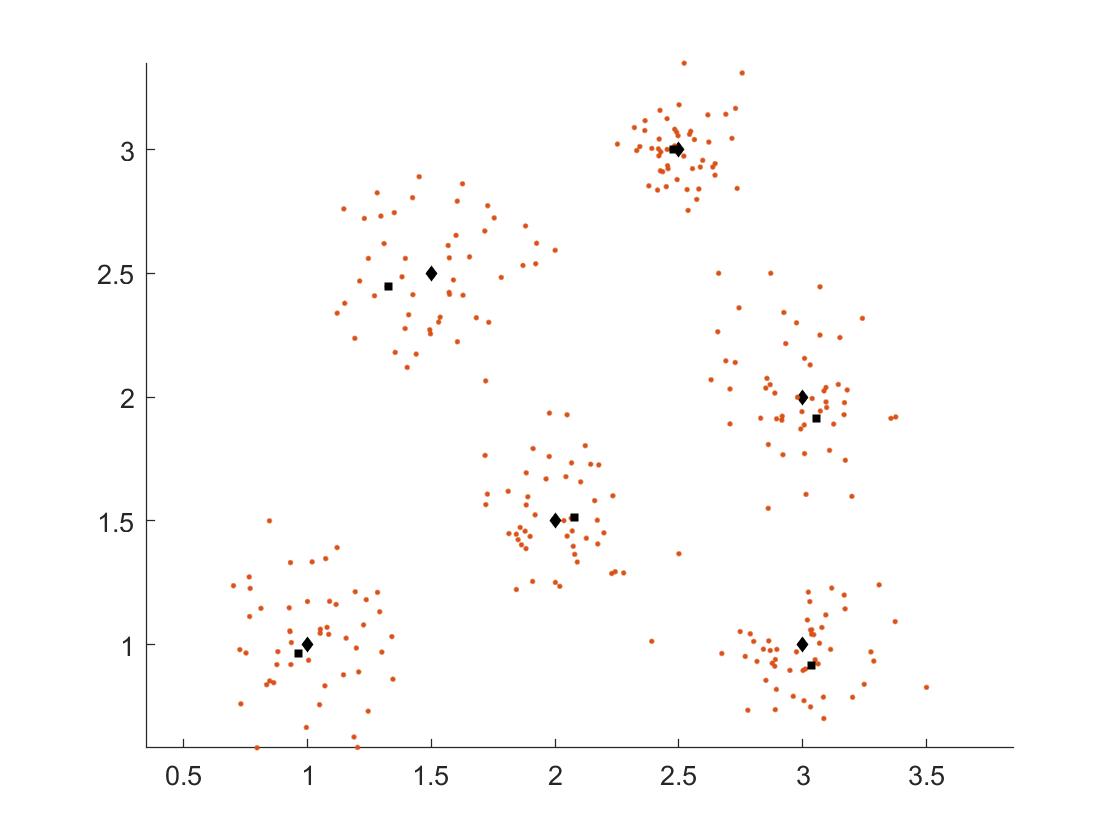}\hfill
  \includegraphics[
 width=.5\textwidth]{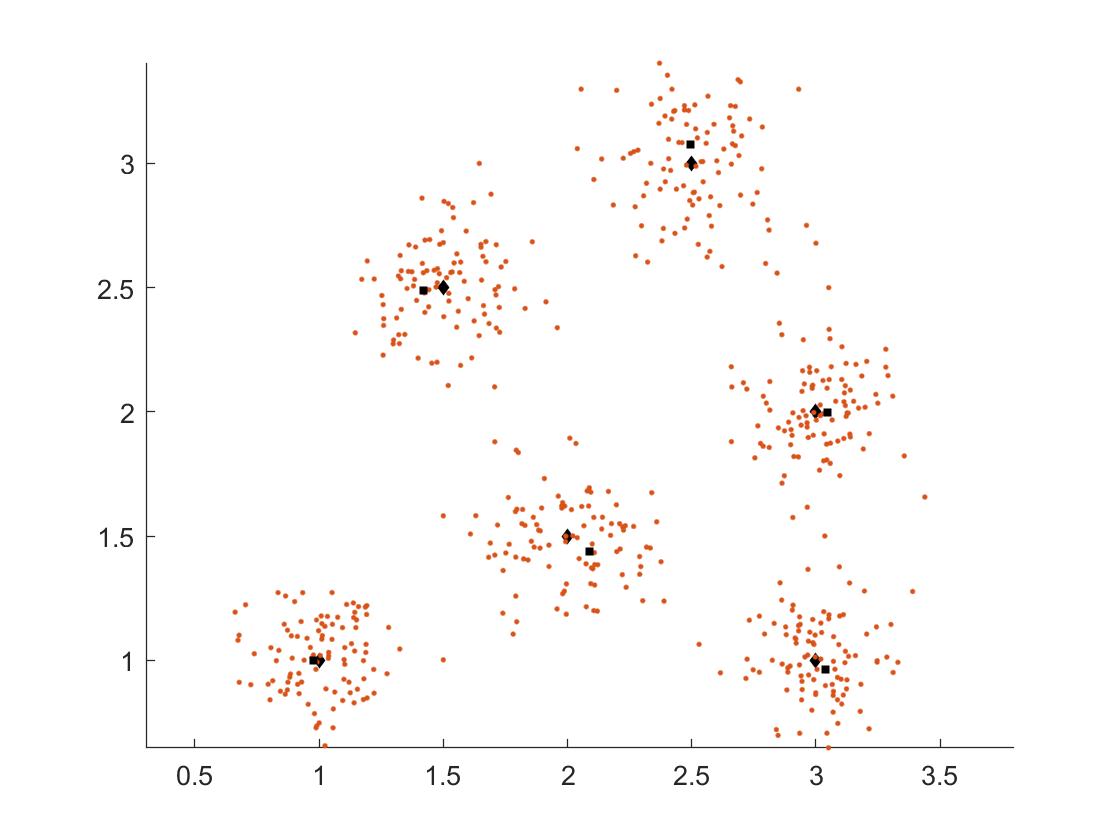}\hfill
 \caption{The performance of Algorithm~\ref{alg:LF} for Example \ref{ex:AlgPer}.
 The left column is for $N_i=50$, and the right column is for $N_i=100$.
 The first row is for $\epsilon = 0.05$, the second row is for $\epsilon=0.1$,
 and the third row is for $\epsilon=0.5$.
 }
 \label{fig:ex:AlgPer}
\end{figure*}
\end{example}

We explore the performance of Algorithm~\ref{alg:LF} for sampling sets $T$
that are not evenly distributed around $S$.
\begin{example}
\label{ex:lf4Tueven}
Let $S$ be the same set given as in Example~\ref{ex:AlgPer}.
Suppose $T$ is a sampling set of $S$ such that for each $i=1,\ldots,6$,
\[
T\subseteq S+a_i [-1,1]^2,\quad |T\cap \{u_i+ a_i [-1,1]^2\}| = b_i,
\]
where $a = (a_1,\ldots, a_6)$ and $b = (b_1,\ldots,b_6)$ are given as
\[
\begin{aligned}
a &= (0.4,\, 0.2,\, 0.6,\, 0.2,\, 0.32,\, 0.4),\\
b &= (50,\, 25,\, 100,\, 30,\, 40,\, 70).
\end{aligned}\]
We apply Algorithm~\ref{alg:LF} for samples generated with the
{\tt MATLAB} function {\tt randn}.
The computational results are summarized as follows.
The computed approximation set is
\[
S^* = \Big\{ \bbm 0.8820\\0.9557\ebm,\bbm 3.0807\\1.7892\ebm,\bbm1.1759\\2.5383\ebm,\bbm 2.3481\\3.0050\ebm,\bbm1.9854\\1.6354\ebm,\bbm 3.0292\\0.8541\ebm\Big\}.
\]
We have that
\[
\|S-S^*\| = 0.3264,\quad \max_{u\in S} f(u) = 0.2147,
\]
where $f(x) = \|\varphi[G](x)\|^2$ is the loss function for $S^*$.
The visualization of Example~\ref{ex:lf4Tueven}
is given in Figure~\ref{fig:lf4Tueven},
where the points in $S$ are plotted
in diamonds and the points in $S^*$ are plotted in squares.
\begin{figure*}[htb]
 \includegraphics[
 width=\textwidth]{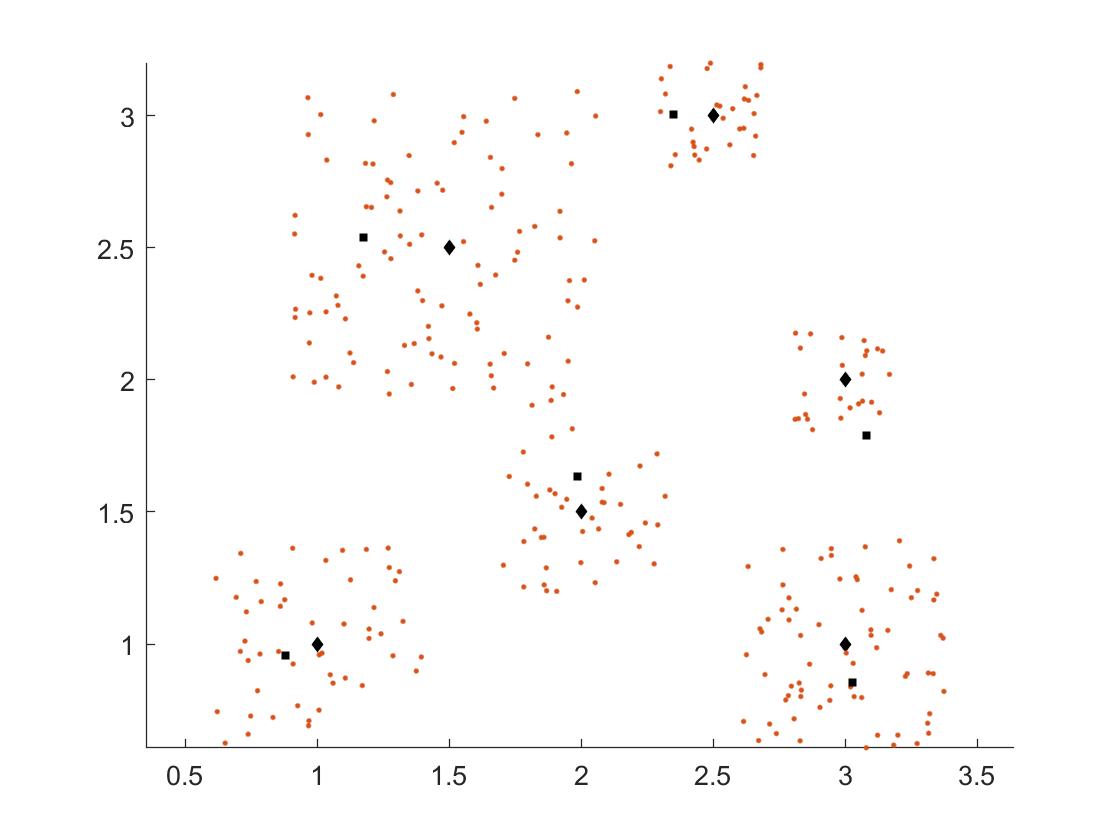}
 \caption{The performance of Algorithm~\ref{alg:LF} for Example~\ref{ex:lf4Tueven}.}
 \label{fig:lf4Tueven}
\end{figure*}
\end{example}

Then, we apply loss functions to study Gaussian mixture models.
For a given sampling set $T$, we compute the finite set $S^*$
and its loss function by Algorithm~\ref{alg:LF}.
The loss function in Section~\ref{sec:standard}
are used, so there are no spurious minimizers.
For a point $v \in T$, apply a nonlinear optimization method
(we use {\tt MATLAB} function {\tt fminunc}) to
minimize $f$ with the starting point $v$.
Once a minimizer $u$ is returned, we cluster $v$
to the group labeled by the point $u\in S^*$.

\begin{example}
\label{ex:GMM}

We use Algorithm \ref{alg:LF} and the transformed simplicial loss functions in
Section~\ref{sec:standard} to learn Gaussian mixture models (GMMs).
Each GMM has parameters $(w_i,\mu_i,\Sigma_i)$,  $i=1,\ldots, k$,
where each weight $w_i >0$, the mean vector $\mu_i \in \re^n$
and the covariance matrix $\Sigma_i \in \mc{S}_{++}^n$
(the cone of real symmetric positive definite $n$-by-$n$ matrices),
such that $w_1+\cdots+w_k=1$. We explore the performance of transformed
simplicial loss functions for two cases
\[
\text{I)}:\, n = 4,\, k\in\{4,5\},\quad
\text{II)}:\, n=5,\, k\in\{3,4\}.
\]
In particular, we compare the results for diagonal Gaussian mixture models
(each $\Sigma_i$ is diagonal) and non-diagonal Gaussian mixture models
(each $\Sigma_i$ is non-diagonal).
For each instance, 1000 samples are generated.
The weights $w_1,\ldots,w_k$ are also computed from sampling:
we first use the {\tt MATLAB} command \texttt{randi}
getting 1000 integers from $[k]$, and then counting each $w_i$
based on the occurrence probability of $i\in[k]$.
We generate each covariance matrix as $\Sigma_i=R^TR$,
for some randomly generated square matrix $R$.
The clustering accuracy rate counts the percentage of
samples belonging to the correct cluster.
We run $10$ instances for each case and give
the average CPU time (in seconds)
consumed by the method and the accuracy rate for all instances.
The computational results are reported in Table~\ref{tab:exmGMM2}.
Algorithm \ref{alg:LF} together with transformed simplicial loss functions
has good performance for both diagonal and non-diagonal Gaussian mixture models.
The clustering accuracy rate is higher for non-diagonal Gaussian mixtures
than that for diagonal ones. In particular, for $(n, k)=(4, 5)$,
the clustering accuracy rate can be as high as 98.92\%.
\begin{table}[htb]
\caption{The computational results for Example~\ref{ex:GMM}.}
\label{tab:exmGMM2}
\centering
\begin{tabular}{cccccc}
\specialrule{.1em}{0em}{0.1em}
& & \multicolumn{2}{c}{Accuracy Rate} & \multicolumn{2}{c}{CPU Time}\\
\hline
$n$ & $k$ &diagonal & non-diagonal & diagonal & non-diagonal\\
\hline
\multirow{2}{*}{4} & 4 & 77.66\% & 85.34\% &  66.14 & 68.28\\
& 5 & 88.73\%  & 98.92\% & 93.32 &  90.76\\
\hline
\multirow{2}{*}{5} & 3 & 80.93\%  & 84.04\% & 73.35 &  75.25\\
& 4 & 82.40\%   &  89.58\%   & 132.88 & 129.19\\
\specialrule{.1em}{0em}{0.1em}
\end{tabular}
\end{table}

\end{example}

\section{Conclusions}

This paper studies loss functions for finite sets.
We give a framework for loss functions.
For a generic finite set $S$,
we show that $S$ can be equivalently given as
the zero set of SOS polynomials with minimum degrees.
When $S$ is the vertex set of a standard simplex,
we show that the given loss function
has no spurious minimizers. For general finite sets,
after a transformation, we can get similar loss functions
that have no spurious minimizers.
When $S$ is approximately given by a sampling set $T$,
we show how to get loss functions for $S$
based on sampling points in $T$.
This can be done by solving a quadratic optimization problem.
Some examples are given
to show the efficiency of the proposed loss functions.

\medskip \noindent
{\bf Acknowledgement}
The authors are partially supported by the NSF grant
DMS-2110780.

\end{document}